\documentclass{amsart}
\usepackage{amsmath}
\usepackage{amssymb}
\usepackage{amsfonts}

\newtheorem{theorem}{Theorem}
\theoremstyle{plain}

\newtheorem{definition}{Definition}

\newtheorem{lemma}{Lemma}

\newtheorem{remark}{Remark}

\numberwithin{equation}{section}
\input{tcilatex}

\begin{document}
\title[On some inequalities]{On some inequalities for $s$-logarithmically
convex functions in the second sense via fractional integrals}
\author{Havva KAVURMACI}
\address{A\u{g}r\i\ \.{I}brahim \c{C}e\c{c}en University, Faculty of Science
and Arts, Department of Mathematics, A\u{g}r\i , Turkey.}
\email{hkavurmaci@agri.edu.tr}
\author{Mevl\"{u}t TUN\c{C}}
\address{Kilis 7 Aral\i k University, Faculty of Science and Arts,
Department of Mathematics, Kilis, 79000, Turkey.}
\email{mevluttunc@kilis.edu.tr}

\begin{abstract}
In this paper, we establish some new Hadamard type inequalities for $s$%
-logarithmically convex functions in the second sense via fractional
integrals by using Lemma 1 which has been proved by Sar\i kaya et al. in the
paper \cite{zeki2}.
\end{abstract}

\maketitle

\section{\protect\bigskip Introduction}

The following result is well known in the literature as Hadamard's
inequality \cite{hdm}.

\begin{theorem}
\bigskip Let $f:I\subset 
\mathbb{R}
\rightarrow 
\mathbb{R}
$ be a convex function on the interval $I$ of real numbers and $a,b\in I$
with $a<b.$ Then%
\begin{equation}
f\left( \frac{a+b}{2}\right) \leq \frac{1}{b-a}\int_{a}^{b}f\left( x\right)
dx\leq \frac{f\left( a\right) +f\left( b\right) }{2}  \label{101}
\end{equation}
\end{theorem}

The following definitions is well known in the literature:

\begin{definition}
\bigskip A function $f:I\rightarrow 
\mathbb{R}
,$ $\emptyset \neq I\subseteq 
\mathbb{R}
,$ where $I$ is a convex set, is said to be convex on $I$ if inequality%
\begin{equation*}
f\left( tx+\left( 1-t\right) y\right) \leq tf\left( x\right) +\left(
1-t\right) f\left( y\right)
\end{equation*}%
holds for all $x,y\in I$ and $t\in \left[ 0,1\right] $.
\end{definition}

\bigskip In \cite{at}, Akdemir and Tun\c{c} were introduced the class of $s$%
-logarithmically convex functions in the first and second sense as the
following:

\begin{definition}
\label{mm}\textit{A function }$f:I\subset 
\mathbb{R}
_{0}\rightarrow 
\mathbb{R}
_{+}$\textit{\ is said to be }$s$-logarithmically convex in the first sense%
\textit{\ if \ \ \ \ \ \ \ \ \ \ \ \ }%
\begin{equation}
f\left( \alpha x+\beta y\right) \leq \left[ f\left( x\right) \right]
^{\alpha ^{s}}\left[ f\left( y\right) \right] ^{\beta ^{s}}  \label{m1}
\end{equation}%
for some $s\in \left( 0,1\right] $, where $x,y\in I$\textit{\ and }$\alpha
^{s}+\beta ^{s}=1.$
\end{definition}

\begin{definition}
\label{mmm}\textit{A function }$f:I\subset 
\mathbb{R}
_{0}\rightarrow 
\mathbb{R}
_{+}$\textit{\ is said to be }$s$-logarithmically convex in the second sense%
\textit{\ if \ \ \ \ \ \ \ \ \ \ \ \ }%
\begin{equation}
f\left( tx+\left( 1-t\right) y\right) \leq \left[ f\left( x\right) \right]
^{t^{s}}\left[ f\left( y\right) \right] ^{\left( 1-t\right) ^{s}}  \label{m2}
\end{equation}%
for some $s\in \left( 0,1\right] $, where $x,y\in I$\textit{\ and }$t\in %
\left[ 0,1\right] $\textit{.}
\end{definition}

\bigskip Clearly, when taking $s=1$ in Definition \ref{mm} or Definition \ref%
{mmm}, then $f$ becomes the standard logarithmically convex function on $I$.

\begin{definition}
Let $f\in L_{1}[a,b].$ The Riemann-Liouville integrals $J_{a^{+}}^{\alpha }f$
and $J_{b^{-}}^{\alpha }f$ of order $\alpha >0$ with $a\geq 0$ are defined by%
\begin{equation*}
J_{a^{+}}^{\alpha }f\left( x\right) =\frac{1}{\Gamma (\alpha )}\underset{a}{%
\overset{x}{\int }}\left( x-t\right) ^{\alpha -1}f(t)dt,\text{ \ }x>a
\end{equation*}%
and%
\begin{equation*}
J_{b^{-}}^{\alpha }f\left( x\right) =\frac{1}{\Gamma (\alpha )}\underset{x}{%
\overset{b}{\int }}\left( t-x\right) ^{\alpha -1}f(t)dt,\text{ \ }x<b
\end{equation*}%
respectively where $\Gamma (\alpha )=\underset{0}{\overset{\infty }{\int }}%
e^{-u}u^{\alpha -1}du.$ Here is $J_{a^{+}}^{0}f(x)=J_{b^{-}}^{0}f(x)=f(x).$
\end{definition}

In the case of $\alpha =1$, the fractional integral reduces to the classical
integral. For some recent results connected with \ fractional integral
inequalities see \cite{zeki2}- \cite{ody}.

In \cite{zeki2}, Sar\i kaya \textit{et. al.} proved the following results
for fractional integrals.

\begin{lemma}
\label{l}Let $f:\left[ a,b\right] \rightarrow 
\mathbb{R}
$ be a differentiable mapping on $\left( a,b\right) $ with $a<b.$ If $%
f^{\prime }\in L\left[ a,b\right] ,$ then the following equality for
fractional integrals holds:%
\begin{eqnarray*}
&&\frac{f\left( a\right) +f\left( b\right) }{2}-\frac{\Gamma \left( \alpha
+1\right) }{2\left( b-a\right) ^{\alpha }}\left[ J_{b^{-}}^{\alpha }f\left(
a\right) +J_{a^{+}}^{\alpha }f\left( b\right) \right] \\
&=&\frac{b-a}{2}\int_{0}^{1}\left[ \left( 1-t\right) ^{\alpha }-t^{\alpha }%
\right] f^{\prime }\left( ta+\left( 1-t\right) b\right) dt.
\end{eqnarray*}
\end{lemma}

\begin{theorem}
\bigskip Let $f:\left[ a,b\right] \rightarrow 
\mathbb{R}
$ be a differentiable mapping on $\left( a,b\right) $ with $a<b.$ If $%
\left\vert f^{\prime }\right\vert $ is convex on $\left[ a,b\right] ,\ $then
the following inequality for fractional integrals holds:%
\begin{eqnarray*}
&&\left\vert \frac{f\left( a\right) +f\left( b\right) }{2}-\frac{\Gamma
\left( \alpha +1\right) }{2\left( b-a\right) ^{\alpha }}\left[
J_{b^{-}}^{\alpha }f\left( a\right) +J_{a^{+}}^{\alpha }f\left( b\right) %
\right] \right\vert \\
&=&\frac{b-a}{2\left( \alpha +1\right) }\left( 1-\frac{1}{2^{\alpha }}%
\right) \left[ \left\vert f^{\prime }\left( a\right) \right\vert +\left\vert
f^{\prime }\left( b\right) \right\vert \right] .
\end{eqnarray*}
\end{theorem}

In the present paper, we will establish several Hermite-Hadamard type
inequalities for the class of functions whose derivatives in absolute value
are $s$-logarithmically convex functions in the first and second sense via
Riemann-Liouville fractional integral.

\section{Hadamard type inequalities for $s$-logarithmically convex functions%
\protect\bigskip}

\begin{theorem}
\label{t}\bigskip Let $I\supset \left[ 0,\infty \right) $ be an open
interval and $f:I\rightarrow \left( 0,\infty \right) $ is differentiable. If 
$f^{\prime }\in L\left[ a,b\right] $ and $\left\vert f^{\prime }\right\vert $
is $s$-logarithmically convex functions in the second sense on $\left[ a,b%
\right] $ for some fixed $s\in \left( 0,1\right] $\textit{\ and }$\mu ,\eta
>0$ with $\mu +\eta =1,$ then the following inequality for fractional
integrals with $\alpha >0$ holds:%
\begin{eqnarray}
&&\left\vert \frac{f\left( a\right) +f\left( b\right) }{2}-\frac{\Gamma
\left( \alpha +1\right) }{2\left( b-a\right) ^{\alpha }}\left[
J_{b^{-}}^{\alpha }f\left( a\right) +J_{a^{+}}^{\alpha }f\left( b\right) %
\right] \right\vert  \label{x0} \\
&\leq &\frac{b-a}{2}\left\{ \int_{0}^{1/2}\mu \left[ \left( 1-t\right)
^{\alpha }-t^{\alpha }\right] ^{\frac{1}{\mu }}dt+\int_{1/2}^{1}\mu \left[
t^{\alpha }-\left( 1-t\right) ^{\alpha }\right] ^{\frac{1}{\mu }}dt\right. 
\notag \\
&&+\left. \eta \times \left\vert f^{\prime }\left( b\right) \right\vert ^{%
\frac{s}{\eta }}\psi \left( \frac{s}{\eta },\frac{s}{\eta }\right) \right\} 
\notag
\end{eqnarray}%
where%
\begin{equation}
\Psi \left( \psi \right) =\left\{ 
\begin{array}{c}
1,\text{ \ \ \ \ \ \ \ }\psi =1, \\ 
\frac{\psi -1}{\ln \psi },\text{ \ \ \ }0<\psi <1%
\end{array}%
\right. \text{ and }\psi \left( u,v\right) =\left\vert f^{\prime }\left(
a\right) \right\vert ^{u}\left\vert f^{\prime }\left( b\right) \right\vert
^{-v},\ \ u,v>0.  \label{2}
\end{equation}
\end{theorem}

\begin{proof}
\bigskip By Lemma \ref{l} and since $\left\vert f^{\prime }\right\vert $ is $%
s$-logarithmically convex functions in the second sense on $\left[ a,b\right]
,$ we have%
\begin{eqnarray}
&&\left\vert \frac{f\left( a\right) +f\left( b\right) }{2}-\frac{\Gamma
\left( \alpha +1\right) }{2\left( b-a\right) ^{\alpha }}\left[
J_{b^{-}}^{\alpha }f\left( a\right) +J_{a^{+}}^{\alpha }f\left( b\right) %
\right] \right\vert  \notag \\
&\leq &\frac{b-a}{2}\int_{0}^{1}\left\vert \left( 1-t\right) ^{\alpha
}-t^{\alpha }\right\vert \left\vert f^{\prime }\left( ta+\left( 1-t\right)
b\right) \right\vert dt  \notag \\
&\leq &\frac{b-a}{2}\int_{0}^{1}\left\vert \left( 1-t\right) ^{\alpha
}-t^{\alpha }\right\vert \left\vert f^{\prime }\left( a\right) \right\vert
^{t^{s}}\left\vert f^{\prime }\left( b\right) \right\vert ^{\left(
1-t\right) ^{s}}dt  \notag \\
&\leq &\frac{b-a}{2}\left\{ \int_{0}^{1/2}\left[ \left( 1-t\right) ^{\alpha
}-t^{\alpha }\right] \left\vert f^{\prime }\left( a\right) \right\vert
^{t^{s}}\left\vert f^{\prime }\left( b\right) \right\vert ^{\left(
1-t\right) ^{s}}dt\right.  \notag \\
&&+\left. \int_{1/2}^{1}\left[ t^{\alpha }-\left( 1-t\right) ^{\alpha }%
\right] \left\vert f^{\prime }\left( a\right) \right\vert ^{t^{s}}\left\vert
f^{\prime }\left( b\right) \right\vert ^{\left( 1-t\right) ^{s}}dt\right\} ,
\label{xf}
\end{eqnarray}%
for all $t\in \left[ 0,1\right] .$ Using the well known inequality $mn\leq
\mu m^{\frac{1}{\mu }}+\eta n^{\frac{1}{\eta }},$ on the right side of (\ref%
{xf}), we have%
\begin{eqnarray*}
&&\left\vert \frac{f\left( a\right) +f\left( b\right) }{2}-\frac{\Gamma
\left( \alpha +1\right) }{2\left( b-a\right) ^{\alpha }}\left[
J_{b^{-}}^{\alpha }f\left( a\right) +J_{a^{+}}^{\alpha }f\left( b\right) %
\right] \right\vert \\
&\leq &\frac{b-a}{2}\left\{ \int_{0}^{1/2}\mu \left[ \left( 1-t\right)
^{\alpha }-t^{\alpha }\right] ^{\frac{1}{\mu }}dt+\int_{0}^{1/2}\eta
\left\vert f^{\prime }\left( a\right) \right\vert ^{\frac{t^{s}}{\eta }%
}\left\vert f^{\prime }\left( b\right) \right\vert ^{\frac{\left( 1-t\right)
^{s}}{\eta }}dt\right. \\
&&+\left. \int_{1/2}^{1}\mu \left[ t^{\alpha }-\left( 1-t\right) ^{\alpha }%
\right] ^{\frac{1}{\mu }}dt+\int_{1/2}^{1}\eta \left\vert f^{\prime }\left(
a\right) \right\vert ^{\frac{t^{s}}{\eta }}\left\vert f^{\prime }\left(
b\right) \right\vert ^{\frac{\left( 1-t\right) ^{s}}{\eta }}dt\right\} \\
&=&\frac{b-a}{2}\left\{ \int_{0}^{1/2}\mu \left[ \left( 1-t\right) ^{\alpha
}-t^{\alpha }\right] ^{\frac{1}{\mu }}dt+\int_{1/2}^{1}\mu \left[ t^{\alpha
}-\left( 1-t\right) ^{\alpha }\right] ^{\frac{1}{\mu }}dt\right. \\
&&+\left. \eta \int_{0}^{1}\left\vert f^{\prime }\left( a\right) \right\vert
^{\frac{t^{s}}{\eta }}\left\vert f^{\prime }\left( b\right) \right\vert ^{%
\frac{\left( 1-t\right) ^{s}}{\eta }}dt\right\}
\end{eqnarray*}%
If $0<\lambda \leq 1,$ $0<u,v\leq 1,$ then%
\begin{equation}
\lambda ^{u^{v}}\leq \lambda ^{uv}.  \label{xx}
\end{equation}%
When $\psi \left( u,v\right) \leq 1,$ by (\ref{xx}), we get that%
\begin{equation}
\int_{0}^{1}\left\vert f^{\prime }\left( a\right) \right\vert ^{\frac{t^{s}}{%
\eta }}\left\vert f^{\prime }\left( b\right) \right\vert ^{\frac{\left(
1-t\right) ^{s}}{\eta }}dt\leq \int_{0}^{1}\left\vert f^{\prime }\left(
a\right) \right\vert ^{\frac{st}{\eta }}\left\vert f^{\prime }\left(
b\right) \right\vert ^{\frac{s\left( 1-t\right) }{\eta }}dt=\left\vert
f^{\prime }\left( b\right) \right\vert ^{\frac{s}{\eta }}\psi \left( \frac{s%
}{\eta },\frac{s}{\eta }\right) .  \label{x11}
\end{equation}%
From (\ref{xf}) to (\ref{x11}), (\ref{x0}) holds.
\end{proof}

\begin{remark}
\bigskip If we take $\alpha =1,$ in Theorem \ref{t}, then the inequality (%
\ref{x0}) become the inequality%
\begin{equation*}
\left\vert \frac{f\left( a\right) +f\left( b\right) }{2}-\frac{1}{b-a}%
\int_{a}^{b}f\left( x\right) dx\right\vert \leq \frac{b-a}{2}\left[ \frac{%
\mu ^{2}}{\mu +1}+\eta \times \left\vert f^{\prime }\left( b\right)
\right\vert ^{\frac{s}{\eta }}\psi \left( \frac{s}{\eta },\frac{s}{\eta }%
\right) \right] .
\end{equation*}
\end{remark}

The corresponding version for powers of the absolute value of the first
derivative is incorporated in the following result:

\begin{theorem}
\bigskip \label{tt} Let $I\supset \left[ 0,\infty \right) $ be an open
interval and $f:I\rightarrow \left( 0,\infty \right) $ is differentiable. If 
$f^{\prime }\in L\left[ a,b\right] $ and $\left\vert f^{\prime }\right\vert $
is $s$-logarithmically convex functions in the second sense on $\left[ a,b%
\right] $ for some fixed $s\in \left( 0,1\right] $\textit{\ and }$\mu ,\eta
>0$ with $\mu +\eta =1$ and $p,q>1,$ then the following inequality for
fractional integrals with $\alpha >0$ holds:%
\begin{eqnarray}
&&\left\vert \frac{f\left( a\right) +f\left( b\right) }{2}-\frac{\Gamma
\left( \alpha +1\right) }{2\left( b-a\right) ^{\alpha }}\left[
J_{b^{-}}^{\alpha }f\left( a\right) +J_{a^{+}}^{\alpha }f\left( b\right) %
\right] \right\vert  \label{89} \\
&\leq &\frac{b-a}{2\left( \alpha p+1\right) ^{\frac{1}{p}}}\left\vert
f^{\prime }\left( b\right) \right\vert ^{s}\left( \psi \left( sq,sq\right)
\right) ^{\frac{1}{q}}  \notag
\end{eqnarray}%
where $1/p+1/q=1,$ and $\psi \left( u,v\right) $ is defined as in (\ref{2}).
\end{theorem}

\begin{proof}
\bigskip \bigskip By Lemma \ref{l} and since $\left\vert f^{\prime
}\right\vert $ is $s$-logarithmically convex functions in the second sense
on $\left[ a,b\right] ,$ we have%
\begin{eqnarray}
&&\left\vert \frac{f\left( a\right) +f\left( b\right) }{2}-\frac{\Gamma
\left( \alpha +1\right) }{2\left( b-a\right) ^{\alpha }}\left[
J_{b^{-}}^{\alpha }f\left( a\right) +J_{a^{+}}^{\alpha }f\left( b\right) %
\right] \right\vert   \label{nn} \\
&\leq &\frac{b-a}{2}\int_{0}^{1}\left\vert \left( 1-t\right) ^{\alpha
}-t^{\alpha }\right\vert \left\vert f^{\prime }\left( a\right) \right\vert
^{t^{s}}\left\vert f^{\prime }\left( b\right) \right\vert ^{\left(
1-t\right) ^{s}}dt  \notag
\end{eqnarray}%
for all $t\in \left[ 0,1\right] .$ Using the well known H\"{o}lder
inequality, on the right side of (\ref{nn}) and making the change of
variable we have%
\begin{eqnarray}
&&\left\vert \frac{f\left( a\right) +f\left( b\right) }{2}-\frac{\Gamma
\left( \alpha +1\right) }{2\left( b-a\right) ^{\alpha }}\left[
J_{b^{-}}^{\alpha }f\left( a\right) +J_{a^{+}}^{\alpha }f\left( b\right) %
\right] \right\vert   \label{90} \\
&\leq &\frac{b-a}{2}\left( \int_{0}^{1}\left\vert \left( 1-t\right) ^{\alpha
}-t^{\alpha }\right\vert ^{p}dt\right) ^{\frac{1}{p}}\left(
\int_{0}^{1}\left\vert f^{\prime }\left( a\right) \right\vert
^{qt^{s}}\left\vert f^{\prime }\left( b\right) \right\vert ^{q\left(
1-t\right) ^{s}}dt\right) ^{\frac{1}{q}}.  \notag
\end{eqnarray}%
It is know that for $\alpha ,t_{1},t_{2}\in \left[ 0,1\right] ,$%
\begin{equation*}
\left\vert t_{1}^{\alpha }-t_{2}^{\alpha }\right\vert \leq \left\vert
t_{1}-t_{2}\right\vert ^{\alpha },
\end{equation*}%
therefore%
\begin{equation}
\int_{0}^{1}\left\vert \left( 1-t\right) ^{\alpha }-t^{\alpha }\right\vert
^{p}dt\leq \int_{0}^{1}\left\vert 1-2t\right\vert ^{\alpha p}dt=\frac{1}{%
\alpha p+1}.  \label{91}
\end{equation}%
Since $\left\vert f^{\prime }\right\vert $ is $s$-logarithmically convex
functions on $\left[ a,b\right] $ and $\psi \left( u,v\right) \leq 1,$ we
obtain%
\begin{equation}
\int_{0}^{1}\left\vert f^{\prime }\left( a\right) \right\vert
^{qt^{s}}\left\vert f^{\prime }\left( b\right) \right\vert ^{q\left(
1-t\right) ^{s}}dt\leq \left\vert f^{\prime }\left( b\right) \right\vert
^{sq}\psi \left( sq,sq\right)   \label{92}
\end{equation}%
From (\ref{90}) to (\ref{92}), (\ref{89}) holds.
\end{proof}

A different approach leads to the following result.

\begin{theorem}
\bigskip \label{ttt}Let $I\supset \left[ 0,\infty \right) $ be an open
interval and $f:I\rightarrow \left( 0,\infty \right) $ is differentiable. If 
$f^{\prime }\in L\left[ a,b\right] $ and $\left\vert f^{\prime }\right\vert
^{q}$ is $s$-logarithmically convex functions in the second sense on $\left[
a,b\right] $ for some fixed $s\in \left( 0,1\right] $\textit{\ and }$\mu
,\eta >0$ with $\mu +\eta =1$ and $q\geq 1,$ then the following inequality
for fractional integrals with $\alpha >0$ holds:%
\begin{eqnarray}
&&\left\vert \frac{f\left( a\right) +f\left( b\right) }{2}-\frac{\Gamma
\left( \alpha +1\right) }{2\left( b-a\right) ^{\alpha }}\left[
J_{b^{-}}^{\alpha }f\left( a\right) +J_{a^{+}}^{\alpha }f\left( b\right) %
\right] \right\vert  \label{50} \\
&\leq &\frac{b-a}{2^{\frac{q-\left( 1-\alpha \right) \left( q-1\right) }{q}}}%
\left( \frac{2^{\alpha }-1}{\alpha +1}\right) ^{1-\frac{1}{q}}\left( \frac{%
\mu ^{2}}{\alpha +\mu }+\eta \left\vert f^{\prime }\left( b\right)
\right\vert ^{\frac{sq}{\eta }}\psi \left( \frac{sq}{\eta },\frac{sq}{\eta }%
\right) \right) ^{\frac{1}{q}}  \notag
\end{eqnarray}%
where $\psi \left( u,v\right) $ is defined as in (\ref{2}).
\end{theorem}

\begin{proof}
\bigskip \bigskip By Lemma \ref{l} and using the well known power mean
inequality, we have%
\begin{eqnarray*}
&&\left\vert \frac{f\left( a\right) +f\left( b\right) }{2}-\frac{\Gamma
\left( \alpha +1\right) }{2\left( b-a\right) ^{\alpha }}\left[
J_{b^{-}}^{\alpha }f\left( a\right) +J_{a^{+}}^{\alpha }f\left( b\right) %
\right] \right\vert \\
&\leq &\frac{b-a}{2}\int_{0}^{1}\left\vert \left( 1-t\right) ^{\alpha
}-t^{\alpha }\right\vert \left\vert f^{\prime }\left( ta+\left( 1-t\right)
b\right) \right\vert dt \\
&\leq &\frac{b-a}{2}\left( \int_{0}^{1}\left\vert \left( 1-t\right) ^{\alpha
}-t^{\alpha }\right\vert dt\right) ^{1-\frac{1}{q}}\left(
\int_{0}^{1}\left\vert \left( 1-t\right) ^{\alpha }-t^{\alpha }\right\vert
\left\vert f^{\prime }\left( ta+\left( 1-t\right) b\right) \right\vert
^{q}dt\right) ^{\frac{1}{q}}
\end{eqnarray*}%
It is easily check that 
\begin{equation*}
\int_{0}^{1}\left\vert \left( 1-t\right) ^{\alpha }-t^{\alpha }\right\vert
dt=\frac{2}{\alpha +1}\left( 1-\frac{1}{2^{\alpha }}\right) .
\end{equation*}%
Since $\left\vert f^{\prime }\right\vert ^{q}$ is $s$-logarithmically convex
and using the well known inequality $mn\leq \mu m^{\frac{1}{\mu }}+\eta n^{%
\frac{1}{\eta }}$, we obtain 
\begin{eqnarray*}
\int_{0}^{1}\left\vert \left( 1-t\right) ^{\alpha }-t^{\alpha }\right\vert
\left\vert f^{\prime }\left( ta+\left( 1-t\right) b\right) \right\vert
^{q}dt &\leq &\int_{0}^{1}\left\vert \left( 1-t\right) ^{\alpha }-t^{\alpha
}\right\vert \left\vert f^{\prime }\left( a\right) \right\vert
^{qt^{s}}\left\vert f^{\prime }\left( b\right) \right\vert ^{q\left(
1-t\right) ^{s}}dt \\
&\leq &\int_{0}^{1}\left\vert 1-2t\right\vert ^{\alpha }\left\vert f^{\prime
}\left( a\right) \right\vert ^{qt^{s}}\left\vert f^{\prime }\left( b\right)
\right\vert ^{q\left( 1-t\right) ^{s}}dt \\
&\leq &\mu \int_{0}^{1}\left\vert 1-2t\right\vert ^{\frac{\alpha }{\mu }%
}dt+\eta \int_{0}^{1}\left\vert f^{\prime }\left( a\right) \right\vert ^{%
\frac{qt^{s}}{\eta }}\left\vert f^{\prime }\left( b\right) \right\vert ^{%
\frac{q\left( 1-t\right) ^{s}}{\eta }}dt.
\end{eqnarray*}%
It is easily check that%
\begin{equation*}
\mu \int_{0}^{1}\left\vert 1-2t\right\vert ^{\frac{\alpha }{\mu }}dt=\mu 
\frac{1}{\frac{\alpha }{\mu }+1}=\frac{\mu ^{2}}{\alpha +\mu }.
\end{equation*}%
Afterwards, when $\psi \left( u,v\right) \leq 1,$ by (\ref{xx}), we get that%
\begin{equation}
\int_{0}^{1}\left\vert f^{\prime }\left( a\right) \right\vert ^{\frac{qt^{s}%
}{\eta }}\left\vert f^{\prime }\left( b\right) \right\vert ^{\frac{q\left(
1-t\right) ^{s}}{\eta }}dt\leq \int_{0}^{1}\left\vert f^{\prime }\left(
a\right) \right\vert ^{\frac{sqt}{\eta }}\left\vert f^{\prime }\left(
b\right) \right\vert ^{\frac{sq\left( 1-t\right) }{\eta }}dt=\left\vert
f^{\prime }\left( b\right) \right\vert ^{\frac{sq}{\eta }}\psi \left( \frac{%
sq}{\eta },\frac{sq}{\eta }\right) .
\end{equation}%
Therefore%
\begin{eqnarray*}
&&\left\vert \frac{f\left( a\right) +f\left( b\right) }{2}-\frac{\Gamma
\left( \alpha +1\right) }{2\left( b-a\right) ^{\alpha }}\left[
J_{b^{-}}^{\alpha }f\left( a\right) +J_{a^{+}}^{\alpha }f\left( b\right) %
\right] \right\vert \\
&\leq &\frac{b-a}{2}\left( \frac{2}{\alpha +1}\left( 1-\frac{1}{2^{\alpha }}%
\right) \right) ^{1-\frac{1}{q}}\left( \mu \int_{0}^{1}\left\vert
1-2t\right\vert ^{\frac{\alpha }{\mu }}dt+\eta \int_{0}^{1}\left\vert
f^{\prime }\left( a\right) \right\vert ^{\frac{qt^{s}}{\eta }}\left\vert
f^{\prime }\left( b\right) \right\vert ^{\frac{q\left( 1-t\right) ^{s}}{\eta 
}}dt\right) ^{\frac{1}{q}} \\
&\leq &\frac{b-a}{2}\left( \frac{2}{\alpha +1}\left( 1-\frac{1}{2^{\alpha }}%
\right) \right) ^{1-\frac{1}{q}}\left( \frac{\mu ^{2}}{\alpha +\mu }+\eta
\left\vert f^{\prime }\left( b\right) \right\vert ^{\frac{sq}{\eta }}\psi
\left( \frac{sq}{\eta },\frac{sq}{\eta }\right) \right) ^{\frac{1}{q}}
\end{eqnarray*}%
which comletes the proof.
\end{proof}

\bigskip

\end{document}